\documentclass{amsart}
\usepackage{a4wide}
\usepackage[latin1]{inputenc}
\usepackage{amssymb,amsfonts,amsmath,amsthm,amscd, epsfig}

\usepackage[all]{xy}

\theoremstyle{plain}
\newtheorem{teo}{Theorem}[section]

\newtheorem{lema}[teo]{Lemma}
\newtheorem{cor}[teo]{Corollary}
\newtheorem{prop}[teo]{Proposition}
\newtheorem{ex}[teo]{Example}
\newtheorem{exes}[teo]{Examples}

\newcommand{\nd}{\noindent}
\newcommand{\vu}{\vspace{.1cm}}
\newcommand{\vd}{\vspace{.2cm}}
\newcommand{\vt}{\vspace{.3cm}}
\newcommand{\vf}{\vspace{.4cm}}

\begin{document}

\title{A Galois-Grothendieck-type correspondence for groupoid actions}

\author{Antonio Paques}
\address{Instituto de Matem\'atica, Universidade Federal do Rio Grande do Sul,
91509-900, Porto Alegre, RS, Brazil} \email{paques@mat.ufrgs.br}
\author{Thaísa Tamusiunas}
\email{trtamusiunas@yahoo.com.br}

\maketitle

\begin{abstract}
In this paper we present a Galois-Grothendieck-type correspondence for
groupoid actions. As an application a Galois-type correspondence is also given.
\end{abstract}

\



\section{Introduction}

S. U. Chase, D. K. Harrison and A. Rosenberg developed in \cite{chase} a Galois theory for commutative ring extensions
$R\supset K$ under the assumption that $R$ is a strongly separable $K$-algebra and the elements of the Galois group $G$
are pairwise strongly distinct $K$-automorphisms of $R$  . Among the main results of that paper, Theorem 2.3 states
a one-to-one correspondence between the subgroups of the group $G$ and the $K$-subalgebras of $R$ which are
separable and $G$-strong.

\vu

The Galois theory due to Grothendieck, in its total generality, is contextualized in the language of schemes
(see \cite{grothendieck}). A version of this theory in the specific context of fields has been presented by A. Dress
in \cite{dress} (see also \cite{BJ}). Dress showed that a simplification of the Galois theory for groups acting on fields
is possible by combining Dedekind's lemma with some elementary facts on $G$-sets, in the case that $G$ is a group.

\vu

Dedekind's lemma states that for a field extension $L$ of a field $K$ the set $Alg_K(A,L)$ of all $K$-algebra homomorphisms
of a $K$-algebra $A$ into $L$ is a linearly independent subset of the $L$-vector space $Hom_K(A,L)$. It turns out that
strongly distinct algebra homomorphisms of separable algebras are a kind of homomorphisms which satisfy a version of Dedekind's lemma.
In \cite{paques}, M. Ferrero and the first author showed that the same approach used by Dress can be adopted in Galois theory for
groups acting on commutative rings, and, as a natural sequel of this method, they obtained some new results.

\vu

The goal of this paper is to develop a Galois theory for groupoids acting on commutative rings using the original viewpoints
of Grothendieck and Dress. We start by introducing a new version of Dedekind's lemma (section 2) we will need for our purposes, and standard
notions and basic facts concerning to groupoid actions on sets and algebras (section 3).
The Galois-Grothendieck-type correspondence for an action $\beta$ of a groupoid $G$ on a $K$-algebra $R$, given in the section 4, establishes an
equivalence between the category of all finite $G$-split sets and the category of all $R$-split $K$-algebras, under the assumption that
$R$ is a $\beta$-Galois extension of $K$. As an application of this result we present in the section 5 a generalization of the Galois-type correspondence given
by Chase, Harrison and Rosenberg in \cite{chase}.

\vu

Throughout, $K$ is a fixed commutative ring with identity  and algebras over $K$ are always commutative and unital. Ring homomorphisms
are assumed to be unitary, and unadorned $\otimes$ means $\otimes_K$.

\vu


\section{Dedekind's Lemma revisited}

We start by recalling that a $K$-algebra $R$ is said to be \emph{separable} if $R$ is a projective $R\otimes R$-module. This
is equivalent to the existence of an element $\upsilon=\sum_ix_i\otimes y_i\in R\otimes R$, which turns out to be an idempotent, unique
such that $\sum_ix_iy_i=1_R$ and $r\upsilon=\upsilon r$, for every $r\in R$.  If, in addition, $R$ is projective and finitely generated as
a $K$-module, we say that $R$ is a {\it strongly separable} $K$-algebra, or, if $R$ is also faithful over $K$,
a {\it strongly separable extension} of $K$. Any faithful, projetive and finitely generated $K$-module is called
\emph{faithfully projective}.

\vu

Let $f, g: T \longrightarrow S$ be ring homomorphisms. We say that $f$ and $g$ are \emph{strongly distinct} if,
for every nonzero idempotent $\pi \in S$, there exists $x \in T$ such that $f(x)\pi \neq g(x)\pi$.

\vu

\begin{lema}\emph{\cite[Lemma 1.2]{paques}} \label{12} Let $T$ be a separable $K$-algebra, and $f: T \to K$
a $T$-algebra homomorphism. Then, there exists a unique idempotent $\pi \in T$ such that $f(\pi) = 1$ and $x\pi = f(x)\pi$, for all $x \in T$.
Furthermore, if $\{f_j\ |\ j\in J\}$ is a nonempty set of pairwise strongly distinct $K$-algebra homomorphisms from $T$ into $K$, then the corresponding
idempotents $\pi_j, j\in J,$ are pairwise orthogonal and $f_i(\pi_j) = \delta_{ij}1_K$, for all $i,j\in J$.

\end{lema}

The next results are slight extensions of similar results given in \cite[Section 2]{paques}.

\vu

\begin{prop}\label{livre} Suppose that $T$ and $R$ are $K$-algebras with $T$ separable over $K$, and $V$ is a nonempty set of
homomorphisms of $K$-algebras $v:T \longrightarrow E_v$, where $E_v = R1_v$ and $\{1_v\}_{v \in V}$ is a
set of nonzero idempotents of $R$. Then, the following statements are equivalent:
\begin{itemize}

\vu

\item [(i)] For each $v \in V$, the elements of $V_v = \{u \in V |$ $1_u = 1_v\}$ are pairwise strongly distinct.

\vu

\item [(ii)] For each $u \in V_v$ there exist $x_{iu} \in E_v$, $y_{iu} \in T$, $1 \leq i \leq m_{u}$, such that
$\sum_{i = 1}^{m_{u}}x_{iu}u'(y_{iu}) = \delta_{u,u'}1_v$, for every $u' \in V_v$.

\vu

\item [(iii)] For each $v \in V$, $V_v$ is free over $E_v$ in $Hom_K(T, E_v)$.
\end{itemize}
\end{prop}

\begin{proof} (i) $\Rightarrow$ (ii) Since $T$ is separable over $K$, for each $v \in V$, $E_v \otimes T$ is
separable over $E_v$. Also, for all $u\in V_v$ the mappings
$$
\begin{array}{cccc}
f_u \ : & \! E_v \otimes T & \! \longrightarrow
& \! E_v, \\
& \! x \otimes y & \! \longmapsto
& \! xu(y)
\end{array}
$$

\nd are pairwise strongly distinct homomorphisms. Then, by  Lemma \ref{12}, there exists
$\pi_u = \sum_{i = 1}^{m_v}x_{iu} \otimes y_{iu} \in E_v \otimes T$ such that $f_{u'}(\pi_u) = \delta_{u,u'}1_v$,
for every $u, u' \in V_v$, and (ii) follows.

\vd

(ii) $\Rightarrow$ (iii) Assume that $V_v'$ is a finite subset of $V_v$ and $\sum_{u'\in V_v'}r_{u'}u' = 0$ in $Hom_K(T, E_v)$,
where $r_{u'} \in E_{u'} = E_v$. Hence, for $u \in V_v'$, we have
\begin{center}
$r_u = (\sum_{u' \in V_v'}\delta_{u,u'}1_v)r_{u'} = \sum_{u' \in V_v'}
(\sum_{i=1}^{m_u}x_{iu}u'(y_{iu}))r_{u'} = \sum_{i=1}^{m_u}x_{iu}(\sum_{u'\in V_v'}u'(y_{iu})r_{u'}) = 0,$
\end{center}
showing that $V_v$ is free over $E_v$.

\vd

(iii) $\Rightarrow$ (i) Immediate.
\end{proof}

\vu

\begin{cor}\label{lemacoro} Assume that $T$ is a strongly separable extension of $K$, $R$ is a $K$-algebra and $V$ is a nonempty set
of homomorphisms of $K$-algebras $v:T \longrightarrow E_v$, where $E_v = R1_v$ and $\{1_v\}_{v \in V}$ is a set of  nonzero idempotents of $R$.
Suppose that for each $v\in V$, the elements of $V_v = \{u \in V |$ $1_u = 1_v\}$ are pairwise strongly distinct. Then,
$\#V_v \leq rank_{K_{\mathfrak{p}}}T_{\mathfrak{p}}$, for every prime ideal $\mathfrak{p}$ of $K$.
\end{cor}

\begin{proof}  It follows from Proposition \ref{livre} that $V_v$ is free over $E_v$ in $Hom_K(T, E_v)$.
Then, we have  via localization that $(V_v)_{\mathfrak{p}}$ is free over $(E_v)_{\mathfrak{p}}$ in
$Hom_{K_{\mathfrak{p}}}(T_{\mathfrak{p}}, (E_v)_{\mathfrak{p}})$, for every prime ideal $\mathfrak{p}$ of $K$.

Furthermore, notice that $T$ is a faithfully projetive $K$-module. So, if $n = rank_{K_{\mathfrak{p}}} T_{\mathfrak{p}}$,
then $T_{\mathfrak{p}} \simeq (K_{\mathfrak{p}})^n$ as $K_\mathfrak{p}$-modules and
$Hom_{K_{\mathfrak{p}}} (T_{\mathfrak{p}}, (E_v)_{\mathfrak{p}}) \simeq ((E_v)_{\mathfrak{p}})^n$ as $(E_v)_{\mathfrak{p}}$-modules.
Consequently, $\#V_v = \#(V_v)_{\mathfrak{p}} \leq n$.
\end{proof}

\vu

\begin{lema}\label{lema23} Assume that $T$ and $R$ are $K$-algebras and $V$ is a non-empty finite set of homomorphisms of $K$-algebras
$v:T \longrightarrow E_v$, where $E_v = R1_v$ and $\{1_v\}_{v \in V}$ is a set of nonzero idempotents of $R$. Suppose that $K$ is isomorphic to
a direct summand of $R$ as $K$-modules and $E_v$ is a faithfully projective $K$-module, for each $v \in V$.
Then, the following statements are equivalent:

\begin{itemize}

\vu

\item [(i)]  $T$ is a strongly separable extension of $K$, for each $v \in V$ the elements of $V_v = \{u \in V |$ $1_u = 1_v\}$
are pairwise strongly distinct and $rank_{K} T = \#V_v$.

\vu

\item [(ii)] $T$ is faithfully projective over $K$, for each $v \in V$ there exist $x_{iv} \in E_v$, $y_{iv} \in T$, $1 \leq i
\leq m_{v}$, such that $\sum_{i = 1}^{m_{v}}x_{iv}u(y_{iv}) = \delta_{u, v}1_v$, for every $u \in V_v$, and $rank_{K} T = \#V_v$.

\vu

\item [(iii)] For each $v \in V$, the mapping $\varphi_v: E_v \otimes T \longrightarrow \prod_{u \in V_v}E_{u}$ given by
$\varphi_v(r \otimes t) = (ru(t))_{u \in V_v}$,  is an isomorphism of $R$-algebras.
\end{itemize}
\end{lema}

\begin{proof} (i)$ \Rightarrow $(ii) Clearly, $T$ is faithfully projective over $K$, and the rest of the assertion
follows from Proposition \ref{livre}.

\vd

(ii)$ \Rightarrow $(iii) Take $v \in V$.
The mapping $\varphi_v$ is clearly an $R$-algebra homomorphism. $\varphi_v$ is also surjective since for any
$r = (r_u)_{u \in V_v} \in \prod_{u \in V_v}E_{u}$, there is $z=\sum_{u \in V_v}\sum_{i=1}^{m_u}r_ux_{iu} \otimes y_{iu} \in E_v \otimes T$
and $\varphi_v(z) = r$. Furthermore, $rank_{E_v}(\prod_{u \in V_v}E_{u}) = rank_{E_v}(E_v)^{\#V_v}= \#V_v = rank_{K} T = rank_{E_v} (E_v \otimes T)$.
Thus,  it follows, by \cite[Corollaire I.2.4]{knus}, that $\varphi_v$ is an isomorphism.

\vd

(iii) $\Rightarrow$ (i) Since, for each $v \in V$, $\varphi_v$ is
an isomorphism, it follows that
$(rank_{K_\mathfrak{p}}(E_g)_\mathfrak{p})(rank_{K_\mathfrak{p}}T_\mathfrak{p})$
$=rank_{K_\mathfrak{p}}(E_g\otimes T)_\mathfrak{p}=rank_{K_\mathfrak{p}}E_g^n=n(rank_{K_\mathfrak{p}}(E_g)_\mathfrak{p})$,
thus  $rank_{K_\mathfrak{p}}T_\mathfrak{p}=n$, for all prime ideal $\mathfrak{p}$ of $K$. Hence,
$rank_KT=n$, so $T$ is faithful over $K$.

\vu

In the sequel we will prove that $T$ is a strongly separable extension of $K$.  It follows from the assumptions on $R$ and $E_v$ that $T\simeq
K\otimes T\simeq K1_v\otimes T$  is isomorphic to a direct summand of $E_v \otimes T \simeq \prod_{u \in V_v}E_{u} = (E_v)^n$,
where $n = \#V_v$. Therefore, $T$ is a finitely generated and projective $K$-module. Furthermore, by \cite[Proposition III.1.7 (c)]{knus}
$(E_v)^n = \prod_{u \in V_v} E_{u}$ is $E_v$-separable. So, by \cite[Proposition III.2.2]{knus}, $T$ is separable over $K$.

\vu

It remains to show that the elements of $V_v$ are pairwise strongly distinct. Given $u \in V_v$, take
$s = (\delta_{l,u}1_l)_{l \in V_v} \in \prod_{u \in V_v}E_{u}$. Then, there exists $z = \sum_{i = 1}^{m_u}r_{iu}
\otimes t_{iu} \in E_v \otimes T$ such that $\varphi_v(z) = s$. Thus, $(\sum_{i=1}^{m_u}r_{iu}
l(t_{iu}))_{l \in V_v} = (\delta_{l,u}1_l)_{l \in V_v}$, that implies $\sum_{i=1}^{m_u}r_{iu}l(t_{iu}) =
\delta_{l,u}1_l$ for each $l \in V_v$, and the assertion follows by Proposition \ref{livre}.
\end{proof}

\vu


\section{Groupoid actions on sets and algebras}

The axiomatic version of groupoid that we adopt in this paper was taken from \cite{lawson}. A \emph{groupoid} is a
nonempty set $G$, equipped with a partially defined binary operation (which will be denoted by concatenation),
where the usual group axioms hold whenever they make sense, that is:
\begin{itemize}
\item [(i)] For every $g, h, l \in G$, $g(hl)$ exists if and only if $(gh)l$ exists and in this case they are equal;
    \item[(ii)] For every $g, h, l \in G$, $g(hl)$  exists if and only if $gh$ and $hl$ exist;
        \item[(iii)] For each $g \in G$, there exist (unique) elements $d(g), r(g) \in G$ such that $gd(g)$ and $r(g)g$
        exist and $gd(g) = g = r(g)g$;
            \item[(iv)] For each $g \in G$ there exists $g^{-1} \in G$ such that $d(g) = g^{-1}g$ and $r(g) = gg^{-1}$.
\end{itemize}

\vu

An element $e\in G$ is called an {\it identity} of $G$ if
$e=d(g)=r(g^{-1})$, for some $g\in G$.  We will denote by $G_0$ the set of all the identities of $G$
and by $G^2$ the set of all the pairs $(g,h)$ such that the product $gh$ is defined.

\vu

The statements of the following lemma are straightforward
from the above definition. Such statements will
be freely used along this paper.

\vu

\begin{lema} \label{lema21}

 Let $G$ be a groupoid. Then,

\begin{itemize}

\item[(i)] for every $g\in G$, the element $g^{-1}$ is unique
satisfying $g^{-1}g=d(g)$ and $gg^{-1}=r(g)$,

\vu

\item[(ii)] for every $g\in G$, $d(g^{-1})=r(g)$ and $r(g^{-1})=d(g)$,

\vu

\item[(iii)] for every $g\in G$, $(g^{-1})^{-1}=g$,

\vu

\item[(iv)] for every $g,h\in G$, $(g,h)\in G^2$ if and only if $d(g)=r(h)$,

\vu

\item[(v)] for every $g,h\in G$, $(h^{-1},g^{-1})\in G^2$ if and only if
$(g,h)\in G^2$ and, in this case, $(gh)^{-1}=h^{-1}g^{-1}$,

\vu

\item[(vi)] for every $(g,h)\in G^2$, $d(gh)=d(h)$ and $r(gh)=r(g)$,

\vu

\item[(vii)] for every $e\in G_0$, $d(e)=r(e)=e$ and $e^{-1}=e$,

\vu

\item[(viii)] for every $(g,h)\in G^2$, $gh\in G_0$ if and only if
$g=h^{-1}$,

\vu

\item[(ix)] for every $g,h\in G$, there exists $l\in G$ such that
$g=hl$ if and only if $r(g)=r(h)$,

\vu

\item[(x)] for every $g,h\in G$, there exists $l\in G$ such that
$g=lh$ if and only if $d(g)=d(h)$.
\end{itemize}
\end{lema}

\vt

Given a groupoid $G$ and $H$ a nonempty subset of $G$, we say that $H$ is a \emph{subgroupoid} of $G$ if it satisfies
the following conditions:
\begin{itemize}
\item [(i)] For every $g, h \in H$, if there exists $ gh$ then $gh \in H$.

\vu

 \item [(ii)] For every $g \in H$, if $g \in H$ then $g^{-1} \in H$.
\end{itemize}
If, in addition, $H_0 = G_0$, we say that $H$ is an \emph{wide} subgroupoid.

\vt

An \emph{action} \emph{of} a groupoid $G$ \emph{on} a nonempty set $X$ is a collection $\gamma$ of
subsets $X_g=X_{r(g)}$  of $X$ and bijections $\gamma_g: X_{g^{-1}} \longrightarrow X_g$ $(g \in G)$ such that:

\begin{itemize}

\item [(i)] $\gamma_e$ is the identity map $ Id_{X_e}$ of $X_e$,  for every $e \in G_0$,

\vu

\item [(ii)] $\gamma_g \circ \gamma_h(x) = \gamma_{gh}(x)$, for every $(g, h) \in G^2$ and $x\in X_{h^{-1}}=X_{(gh)^{-1}}$.
\end{itemize}
In this case, we also say that $X$ is a $G$-\emph{set}. If, in addition, the union of the subsets
$X_e$, $e\in G_0$, is disjoint and equal to $X$ (shortly $X=\dot{\bigcup}_{e\in G_0}X_e$) we say that $X$ is a \emph{G-split set}.

\vu

\begin{ex} \label{exa32} \em A groupoid $G$ is a $G$-split set. In fact, for $X = G$, take $X_g = r(g)G = \{r(g)l \ |\ r(l) = r(g)\}=X_{r(g)}$ and
$\gamma_g: X_{g^{-1}} \to X_g$ given by $\gamma_g(d(g)l)=gd(g)l\, (= gl=r(g)gl)$, for all $g\in G$. Notice that $G=\dot{\bigcup}_{e\in G_0}X_e$ by
construction.
\end{ex}

\vu

\begin{ex}\label{ex33} \em Consider $H$ an wide  subgroupoid of $G$. Take the equivalence relation $\equiv_H$ defined by: for every $a, b \in G$,
$a \equiv_H b$ if and only if  there exists  $b^{-1}a$  and $b^{-1}a \in H$ . Notice that
$g=gd(g)\in gH=\{gh\ |\ r(h)=d(g)\}$, for every $g\in G$, for $H$ is wide. Then, the set $\frac{G}{H} = \{gH \ |\ g \in G\}$ is a $G$-split set.
Indeed, for $X = \frac{G}{H}$, it is enough to take  $X_{g}= \{lH \in \frac{G}{H}|$ $r(l) = r(g)\}=X_{r(g)}$ and to define $\gamma_{g}: X_{g^{-1}} \to X_g$
by $\gamma_{g}(lH) = glH$, for all $g\in G$. As in the previous example, also here $\frac{G}{H} = \dot{\bigcup}_{e \in G_0} X_{e}$ by construction.
\end{ex}

\vt

An \emph{action} of a groupoid $G$ \emph{on} a $K$-algebra $R$ \cite{bagio} is  a collection $\beta$ of ideals $E_g = E_{r(g)}$  of $R$  and algebra isomorphisms
$\beta_g: E_{g^{-1}} \to E_g$ ($g\in G)$, such that $R$ is a $G$-set via $\beta$. In this case, the set
$$R^\beta:=\{r\in R\ |\ \beta_g(rx)=r\beta_g(x),\ \text{for all}\ g\in G \ \text{and}\  x\in E_{g^{-1}}\}$$ is indeed a $K$-subalgebra of
$R$, called the \emph{subalgebra of the invariants of} $R$ \emph{under the action} $\beta$. If each
$E_g$ is unital, with identity element $1_g$,  then it is immediate to see that $r\in R^\beta$ if and only if $\beta_g(r1_{g^{-1}})=
r1_g$, for all $g\in G$.

\vt

Let $R$, $G$ and $\beta=\{\beta_g:E_{g^{-1}}\to E_g\}_{g\in G}$ be as above.
Accordingly \cite{bagio}, the skew groupoid ring
$R\star_\beta G$ corresponding to $\beta$ is defined as the direct
sum
$$R\star_\beta G=\bigoplus_{g\in G}E_g\delta_g$$ in which the
$\delta_g$'s are symbols, with the usual addition, and
multiplication determined by the rule
\[
(x\delta_g)(y\delta_h)=
\begin{cases}
x\beta_g(y)\delta_{gh} &\text{if $(g,h)\in G^2$}\\
0  &\text{otherwise},
\end{cases}
\]
for all $g, h\in G$, $x\in E_g$ and $y\in E_h$. It is straightforward to check that this multiplication is well defined
and that $R\star_\beta G$ is associative. If $G_0$ is finite and each $E_e$, $e\in G_0$, is unital, then $R\star_\beta G$ is also unital \cite{FP},
with identity element given by $\sum_{e\in G_0}1_e\delta_e$, where $1_e$ denotes the identity element of $E_e$.

\vf

Hereafter, in this section,

\begin{itemize}

\item $G$ is a finite groupoid,

\vu

\item $\gamma = \{\gamma_g: X_{g^{-1}}\to X_g\}_{g \in G}$ is an action of $G$ on
a fixed nonempty and finite set $X$ such that $X = \dot{\bigcup}_{e \in G_0} X_e$, that is, $X$ is a finite $G$-split set.

\vu

\item and $\beta = \{\beta_g: E_{g^{-1}}\to E_g\}_{g \in G}$ is an action of $G$ on a fixed faithful $K$-algebra $R$
such that  each  $E_e$ ($e\in G_0$) is unital with identity element $1_e$,  $R = \bigoplus_{e \in G_0}E_e$, and $R^\beta=K.$
\end{itemize}

\vd

In this context, any left $R\star_\beta G$-module $M$ is also an $R$-module via the imbedding
$r\mapsto\sum_{e\in G_0}r1_e\delta_e$, for all $r\in R$. We put
$$M^G=\{x\in M\ |\ (1_g\delta_g)x=1_gx,\, \text{for all}\, g\in G \}$$
to denote the $K$-module of the \emph{invariants of} $M$ \emph{under} $G$. Notice that the $K$-algebra $R$ is also a left
$R\star_\beta G$-module via the action $(r_g\delta_g)x=r_g\beta_g(x1_{g^{-1}})$, for all $x\in R$, $g\in G$ and $r_g\in E_g$, and
$R^G=R^\beta=K$.

\vd

Now, consider the set
$$Map(X, R) =\{f: X \to R\ |\  f(X_g) \subseteq E_g,\ \text{for all}\, g\in G \},$$ which clearly  is an $R$-algebra (in particular,
a $K$-algebra) under the usual pointwise operations, whose identity element is $\sum_{e \in G_0} 1'_e$, where $1'_g$ is defined by
\[
1'_{g}(x) =
\begin{cases}
1_g\,& \mbox{if} \quad x \in X_g\\
0, & \mbox{if} \quad x \notin X_g,\\
\end{cases}
\]
for every $g\in G$.

Furthermoremore, it is straightforward to check that

\begin{itemize}
\item $M_g=Map(X, R)_g = \{f \in Map(X, R)\ |\ f(X_h) = 0,\,\text{if}\,\, X_h \neq X_g\}$ is an ideal of $Map(X,R)$ with
identity element $1'_g$;

\vu

\item $M_g=M_{r(g)}$;

\vu

\item $\alpha_g:M_{g^{-1}}\to M_g$, given by
\[
\alpha_g(f1'_{g^{-1}})(x) =
\begin{cases}
\beta_g \circ f1'_{g^{-1}} \circ \gamma_{g^{-1}}(x) & \text{if}\,\ x\in X_g\\
0 & \text{otherwise},
\end{cases}
\]
is an isomorphism of $K$-algebras;

\vu

\item $\alpha=\{\alpha_g:M_{g^{-1}}\to M_g\}_{g\in G}$
is an action of $G$ on $Map(X,R)$;

\vu

\item $Map(X,R)=\bigoplus_{e \in G_0} M_e$;

\vu

\item $Map(X,R)$ is a left $R\star_\beta G$-module via the action $(r_g\delta_g)f=r_g\alpha_g(f1'_{g^{-1}})$.
\end{itemize}

\vd

We will denote by $A(X)$ the $K$-subalgebra of the invariants of $Map(X,A$) under $\alpha$, as well as under $G$, that is,
$A(X) = Map(X, R)^{\alpha} = \{f \in Map(X, R)\ |\ \alpha_g(f1'_{g^{-1}}) = f1'_g,\, \text{for all}\,  g \in G\}=Map(X,R)^G$. Notice
that if $f\in A(X)$, then $\beta_g(f(x))=f(\gamma_g(x))$, for every $x\in X_{g^{-1}}$.

\vt

For $g\in G$ and every $x\in X_g$ set $E_x=E_g$. For $g\in G$ and $x \in X$, let $\rho_x:A(X)\to E_x$ be the algebra homomorphism given
by $\rho_x(f) = f(x)$, for every $f \in A(X)$. Set  $V_g(X):=\{\rho_x\ |\ x\in X_g\}$. Clearly, $V_g(X)=V_{r(g)}(X)$.

\vu

\begin{lema}\label{xconj} Assume that $K$ is a direct summand of $R$ as $K$-modules and $E_g$ is a faithfully
projective $K$-module, for each $g \in G$. Then the following conditions are equivalent:

\vu

\begin{itemize}

\item [(i)] For every $g \in G$, the elements of $V_g(X)$ are pairwise strongly distinct,
$rank_{K}A(X) = \#V_g(X)$ and $A(X)$ is a strongly separable extension of $K$;

\vu

\item [(ii)] For every $g\in G$, the map $\varphi_g: E_g \otimes A(X) \to \prod_{x \in X_g}E_{x}$, given by
$\varphi_g(r \otimes f)= (rf(x))_{x \in X_g}$, is an isomorphism of $R$-algebras.
\end{itemize}
\end{lema}

\begin{proof}  It is an immediate consequence of Lemma \ref{lema23}.\end{proof}

\vu

Following \cite{bagio} $R$ is a \emph{$\beta$-Galois extension of} $R^\beta=K$  if there exist elements $r_i, s_i\in R$, $1\leq i\leq m$,  such that
$\sum_{1\leq i\leq m}x_i\beta_g(s_i1_{g^{-1}})=\delta_{e,g}1_e$, for all $e\in G_0$ and $g\in G$. The elements $x_i,y_i$ are called the
\emph{$\beta$-Galois coordinates of} $R$
over $R^\beta$. It is immediate to see that, in this case, the \emph{trace map}
$$t_\beta:R\to R,\quad\text{ given by}\quad t_\beta(r)=\sum_{g\in G}\beta_g(r1_{g^{-1}}),$$ is a $K$-linear map, and $t_\beta(R)=K$ by
\cite[Lemma 4.2 and Corollary 5.4]{bagio}.
Hence, $K$ is a direct summand of $R$ as $K$-modules.

\vu

\begin{lema}\label{xconj2} Assume that $R$ is a $\beta$-Galois extension of $K$. Then, for each $g\in G$, the map $\varphi_g: E_g \otimes A(X)
\to \prod_{x \in X_g}E_{x}$, given by $\varphi_g(r \otimes f) = (rf(x))_{x \in X_g}$, is an isomorphism of $R$-algebras.
\end{lema}

\begin{proof}  Since $Map(X,R)^G=A(X)$, it follows from \cite[Theorem 5.3]{bagio} that the map  $\mu: R \otimes A(X) \to Map(X, R)$ given by
$\mu(r \otimes f) = rf$ is an isomorphism of $R$-algebras, which clearly induces an isomorphism $\mu_g: E_g \otimes A(X)) \to Map(X_g,E_g)$.
On the other hand, $Map(X_g,E_g)\simeq \prod_{x \in X_g}E_{x}$, as $R$-algebras, via the map $\eta_g:f\mapsto (f(x))_{x \in X_g}$.
Since $\varphi_g=\eta_g\mu_g$, the result follows.
\end{proof}

\vu


\section{The Galois-Grothendieck-type Correspondence}

We start recalling that $G$, $R$, $X$, $\beta$ and $\gamma$ are as in the previous section.
Let $V(X)=\bigcup_{e\in G_0}V_e(X)=\{\rho_x\ |\ x\in X_e,\ e\in G_0\}=\{\rho_x\ |\ x\in X_g,\ g\in G\}$.

\vd

Let $Y$ and $W$ be $G$-sets via the actions $\varepsilon = \{\varepsilon_g: Y_{g^{-1}} \to Y_g\}_{g \in G}$ and
$\vartheta = \{\vartheta_g: W_{g^{-1}} \to W_g\}_{g \in G}$, respectively. A map $\psi: Y \to W$ is said an \emph{isomorphism of  $G$-sets} if
the following conditions are satisfied:

\begin{itemize}
\item [(i)] $\psi$ is a bijection;

\vu

\item [(ii)] $\psi(Y_g) = W_g$, for all $g \in G$;

\vu

\item[(iii)] $\psi(\varepsilon_g(y)) = \vartheta_g(\psi(y))$, for all $y \in Y_{g^{-1}}$ and $g \in G$.
\end{itemize}

\vu

\begin{lema}\label{lema1} Assume that $R$ is a $\beta$-Galois extension of $K$. Then:
\begin{itemize}
\item[(i)] $V(X)$ is a $G$-split set;

\vu

\item[(ii)]  The elements of $V_g(X)$ are pairwise strongly distinct, for every $g \in G$,;

\vu

\item[(iii)] The map $\omega: X \to V(X)$, given by $\omega(x) = \rho_x$, is an isomorphism of $G$-sets.
\end{itemize}
\end{lema}

\begin{proof} (i) Take $\sigma = \{\sigma_g: V_{g^{-1}}(X) \to V_g(X)\}_{g \in G}$, where $\sigma_g(\rho_x)(f) =
\beta_g(f(x))$, for every $x \in X_{g^{-1}}$. Observe that $f \in A(X)$, hence $\sigma_g(\rho_x)(f) = \beta_g(f(x)) =
f(\gamma_g(x)) = \rho_{\gamma_g(x)}(f)$ and, consequently,  $\sigma_g(\rho_x) \in V_g(X)$, showing that the map $\sigma_g$ is well-defined.
Moreover, $\sigma_g$ is a bijection with inverse $\sigma_{g^{-1}}$, for every $g\in G$. It is immediate to
check that $\sigma$ is an action of $G$ on $V(X)$, and $V(X)=\dot{\bigcup}_{e\in G_0}V_g(X)$ by construction.

\vd

(ii) It follows from Lemma \ref{xconj2} that, for every $g\in G$, the map $\varphi_g: E_g\otimes A(X) \to \prod_{x \in X_g}E_{x}$,
given by $\varphi_g(r \otimes f) = (rf(x))_{x \in X_g}$, is an isomorphism of $R$-algebras.  Thus, for each $x \in X_g$, there exist
$r_{ix} \in E_g$ and $f_{ix} \in A(X)$, $1 \leq i \leq m_x$, such that $(\sum_{i=1}^{m_x}r_{ix}f_{ix}(y))_{y \in X_g} = (\delta_{x,y}1_g)_{y \in X_g}$.
Hence,  $\sum_{i=1}^{m_x}r_{ix}\rho_y(f_{ix}) = \sum_{i=1}^{m_x}r_{ix}f_{ix}(y) = \delta_{x,y}1_g$, for every $y \in X_g$, and the
assertion follows by Proposition \ref{livre}.

\vd

(iii) Consider the surjective map  $\omega_g: X_g \to V_g(X)$ given by $\omega_g(x) = \rho_x$, for every $x \in X_g$.
Indeed, $\omega_g$ is a bijection. If $\rho_x = \rho_y$, for $x, y \in X_g$, then $f(x) = f(y)$, for every $f \in A(X)$.

\vu

On the other hand, the map $\eta_g: Map(X_g,E_g) \to \prod_{x \in X_g}E_{x}$, given by $\eta_g(f) = (f(x))_{x \in X_g}$, is an isomorphism of
$R$-algebras, whose inverse is the map $\eta'_g: \prod_{x \in X_g}E_{x} \to Map(X_g,E_g)$ given by $\eta'_g(r)(x) = r_x$,
where $r = (r_x)_{x \in X_g} \in \prod_{x \in X_g}E_{x}$. Furthermore, the map $\varphi_g: E_g \otimes A(X)\to \prod_{x \in X_g}E_{x}$,
given by $\varphi_g(r \otimes f) = (rf(x))_{x \in X_g}$, is also an isomorphism of $R$-algebras, by Lemma \ref{xconj2}.

\vu

Thus, $E_g \otimes A(X) \simeq \prod_{x \in X_g}E_{x} \simeq Map(X_g,E_g)$, and so, for every
$p \in Map(X_g,E_g)$, there exists $\lambda=\sum_{1\leq i\leq m}r_i\otimes f_i \in E_g\otimes A(X)$ such that $p =
\eta'_g \circ \varphi_g(\lambda)$. Consequently, $$p(x) =(\eta'_g\circ \varphi_g (\lambda))(x) =
\eta'_g((\sum_{1\leq i\leq m}r_if_i(z))_{z \in X_g})(x) = \sum_{1\leq i\leq m}r_if_i(x) =
\sum_{1\leq i\leq m}r_if_i(y) = p(y),$$ for every $p\in Map(X_g,E_g)$. So, $x = y$.

\vu

Therefore,  the map $\omega: X\to V(X)$, given by $\omega(x) = \omega_g(x)$ if $x \in X_g$,
is also a bijection, and $\omega(X_g)=V_g(X)$.

\vu

Finally, $\omega$ commutes with  the actions $\sigma$ and $\gamma$. Indeed, for $x \in X_{g^{-1}}$ and $f\in A(X)$, we have
$$\omega(\gamma_g(x))(f) = \rho_{\gamma_g(x)}(f) = f(\gamma_g(x)) = \beta_g(f(x)) = \sigma_g(\rho_x)(f)
 = \sigma_g(\omega(x))(f),$$ which concludes the proof.
\end{proof}

\vu

For any $K$-algebras $B$ and $C$, we will denote by $Alg_K(B,C)$ the set of all $K$-algebra homomorphisms from $B$ into $C$.

\vu

\begin{lema}\label{lema2} Let $B$ be a $K$-algebra and $g\in G$. Suppose that $E_g$ is faithfully projective and
there exists an isomorphism of $E_g$-algebras $\varphi_g: E_g \otimes B \to (E_g)^{n_g}$, $n_g \geq 1$. Then:

\vu

\begin{itemize}

\item[(i)] $B$ is faithfully projective over $K$ with constant rank $n_g$;

\vu

\item[(ii)] $B$ is a strongly separable extension of $K$;

\vu

\item [(iii)] There exist $\varphi_{(g,1)},\ldots, \varphi_{(g,n_g)} \in Alg_K (B, E_g)$ such that $\varphi_g(r \otimes b) =
(r\varphi_{(g,i)}(b))_{1\leq i\leq n}$ for every $r \in E_g$ and $b \in B$;

\vu

\item [(iv)] The elements of $V_g(B) = \{\varphi_{(g,i)}|$ $1 \leq i \leq n_g\}$ are pairwise strongly distinct;

\vu

\item [(v)] $V_g(B) = Alg_K(B, E_g)$ whenever the elements of $Alg_K(B, E_g)$ are pairwise strongly distinct.
\end{itemize}
\end{lema}

\begin{proof}  The assertions (i) and (ii) follows by the same arguments used in the
proof  of Lemma \ref{lema23}((iii)$\Rightarrow$(i)).

\vd

(iii) Denote by $\eta_g: B \to E_g \otimes B$ the map given by $b\mapsto 1_g \otimes b$, and by
$\pi_{(g,i)}: (E_g)^n \to E_g$ the $i^{th}$-projection, for every $1\leq i\leq n_g$. Clearly, the maps $\varphi_{(g,i)}:= \pi_{(g,i)}\varphi_g\eta_g$
are in $Alg_K(B,E_g)$ and it is easy to see that  $\varphi_g(r \otimes b)=(r\varphi_{(g,i)}(b))_{1\leq i\leq n_g} $, for all $r\in E_g$ and $b\in B$.

\vd

(iv) Since $\varphi_g$ is an isomorphism, for each $1 \leq i \leq n_g$, there exist $r_{il} \in E_g$ and $b_{il} \in B$,
$1 \leq l \leq m_g$, such that $\varphi_g(\sum_{l=1}^{m_g}r_{il} \otimes b_{il}) = (\sum_{l=1}^{m_g}r_{il}\varphi_{(g,j)}(b_{il}))_{1\leq j\leq n_g}
= (\delta_{i,j}1_g)_{1\leq j\leq n_g}$, that is, $\sum_{l=1}^{m_g}s_{il}\varphi_{(g,j)}(b_{il}) = \delta_{i,j}1_g$, for every $1 \leq j \leq n_g$.
Consequently, the elements of $V_g(B)$ are pairwise strongly distinct, by (ii) and Proposition \ref{livre}.

\vd

(v) Suppose that the elements of $Alg_K(B, E_g)$ are pairwise strongly distinct. Then, by (i), (ii) and Corollary \ref{lemacoro},
$\# Alg_K(B, E_g) \leq rank_KB=n_g=\#V_g(B)\leq\# Alg_K(B,E_g)$. Thus,
$V_g(B) = Alg_K(B, E_g)$.
\end{proof}

\vu

The next lemma provide us a necessary and sufficient condition for the set  $V(B) = {\bigcup}_{e \in G_0} V_e(B)$ to be a $G$-set. Again here,
this union is disjoint and finite by construction.

\vu

\begin{lema} Let $B$, $E_g$, $\varphi_g$ and $V_g(B)$ $(g\in G)$, be as in Lemma \ref{lema2}. Then the following conditions are equivalent:

\begin{itemize}

\vu

\item [(i)] $V(B)$ is a $G$-set via $\xi = \{\xi_g: V_{g^{-1}}(B) \to V_g(B)\}_{g \in G}$, with $\xi_g(\varphi_{(g^{-1},i)})(b) =
\beta_g(\varphi_{(g^{-1},i)}(b))$, for every $b\in B$;

\vu

\item [(ii)] For every $g, h \in G$ with $r(g) = r(h)$, given $\varphi_{(g^{-1},i)}$ and $\varphi_{(g^{-1},j)}$ in $V_{g^{-1}}(B)= V_{h^{-1}}(B)$,
the elements $\xi_g(\varphi_{(g^{-1},i)})$ and $\xi_h(\varphi_{(g^{-1},j)})$ are strongly distinct.
\end{itemize}
\end{lema}

\begin{proof} (i) $\Rightarrow$ (ii) It is enough to notice that if $r(g)=r(h)$ then $V_g(B)=V_h(B)$. Now, the assertion follows from Lemma \ref{lema2}(iv).

\vd

(ii) $\Rightarrow$ (i) It is enough to show that each $\xi_g$, $g\in G$, is a bijection for the conditions (i)-(ii) of the definiton of a groupoid
action are straightforward. Also, each $\xi_g$  is injective by construction, thus it is enough to prove that it is surjective.

\vu

We start by noticing that the elements of $V_{g^{-1}}(B)$ are pairwise strongly distinct, by Lemma \ref{lema2}. Consequently,
the elements of $\xi_g(V_{g^{-1}}(B))$ are pairwise strongly distinct and it follows from the assumption that also the elements of
$Y_g(B) = \bigcup_{\{h\in G|r(h)=r(g)\}}\xi_h(V_{g^{-1}}(B))$ are pairwise strongly distinct.

\vu

Clearly, $Y_g(B)\subseteq V_g(B)$, and noting that $r(r(g))=r(g)$ and $V_g(B)=V_{r(g)}(B)=V_{r(g)^{-1}}(B)=\xi_{r(g)}(V_{r(g)^{-1}}(B))$,
 we have that $V_g(B)\subseteq Y_g(B)$, for every $g\in G$.

\vu

Furthermore, $\xi_g(V_{g^{-1}}(B)) \subseteq Y_{g}(B)=V_g(B)$ and by Lemma \ref{lema2} $\#\xi_g(V_{g^{-1}}(B))=\#V_{g^{-1}}(B)=n_{g^{-1}}=rank_KB=n_g =\#V_g(B)$.
Hence, $\xi_g(V_{g^{-1}}(B)) = V_g(B)$, and $\xi_g$ is a bijection.
\end{proof}

\vd

Assume that $S = \bigoplus_{j=1}^n S_j$ is a $K$-algebra, where $S_j = S1_j$ and $\{1_j\}_{1 \leq j \leq n}$ are pairwise orthogonal central idempotents in $S$,
for some $n \geq 1$. An $K$-algebra $T$ is said to be $S$-\emph{split} if:

\begin{itemize}

\vu

\item [(i)] For each $1 \leq j \leq n$, there exists an isomorphism of $K$-algebras $\phi_j: S_j \otimes T \to (S_j)^m$, for some given $m \geq 1$;

\vu

\item [(ii)] $V(T) = \bigcup_{j=1}^n V_j(T)$ is a $G$-set, where $V_j(T)$ is defined as in Lemma \ref{lema2}.
\end{itemize}

\vu

\nd Notice that (i) is equivalent to say that $S \otimes T \simeq S^m$ and, in particular, $V(T)$ is a finite $G$-split set.

\vu

\begin{lema}\label{bijgro} Let $B$, $E_g$, $\varphi_g$ and $V_g(B)$  $(g\in G)$ be as in Lemma \ref{lema2}. Assume that $R$ is a $\beta$-Galois extension
of $K$ and $V(B)$ is a $G$-set via
$\xi = \{\xi_g: V_{g^{-1}}(B) \to V_g(B)\}_{g \in G}$. Then, the mapping $\nu: B \to A(V(B))$, given by $\nu(b)(\varphi_{(g,i)}) = \varphi_{(g,i)}(b)$,
for $b \in B$ and $\varphi_{(g,i)} \in V(B)$, is an isomorphism of $K$-algebras.
\end{lema}

\begin{proof} We start by checking that $\nu$ is a well defined. Indeed, for $g \in G$, $b \in B$ and $\varphi_{(g,i)} \in V(B)$,
we have
$$
\begin{array}{ccl}
\alpha_g(\nu(b)1_{g^{-1}}')(\varphi_{(g,i)}) & \! = & \! \beta_g \circ \nu(b)1_{g^{-1}}' \circ \xi_{g^{-1}}(\varphi_{(g,i)}) =
\beta_g(\nu(b)(\xi_{g^{-1}}(\varphi_{(g,i)}))1_{g^{-1}})  \\
& \! = & \! \beta_g (\xi_{g^{-1}}(\varphi_{(g,i)})(b)1_{g^{-1}}) = \beta_g(\beta_{g^{-1}}(\varphi_{(g,i)}(b)1_g)1_{g^{-1}})  \\
& \! = & \! \beta_{r(g)}(\varphi_{(g,i)}(b)1_{r(g)}) = \varphi_{(g,i)}(b)1_{r(g)} \\
& \! = & \! \varphi_{(g,i)}(b)1_g = \nu(b)(\varphi_{(g,i)})1_g = \nu(b)1_g'(\varphi_{(g,i)}),
\end{array}
$$
showing that $\nu(b) \in A(V(B))$. Clearly, $\nu$ is an algebra homomorphism. It remains to check that it is a bijection.

\vu

Given $a, b \in B$, if $a \neq b$, then $\varphi_g(1_g \otimes a) \neq \varphi_g(1_g \otimes b)$, since for each $g \in G$, $E_g$ is faithful
over $K$ and $\varphi_g$ is an isomorphism.
Thus, there exists $1\leq i\leq n_g $ such that $\nu(a)(\varphi_{(g,i)}) = \varphi_{(g,i)}(a) \neq \varphi_{(g,i)}(b) = \nu(b)(\varphi_{(g,i)})$.
So, $\nu(a) \neq \nu(b)$ and $\nu$ is injective.

\vd

By Lemmas \ref{xconj2} and \ref{lema2}, the $K$-algebras $A(V(B))$ and $B$ are faithfully projective and separable, and $rank_K A(V(B)) = \#V_g(B)=rank_KB$.
Since, $\nu(B)\simeq B$ as $K$-algebras, it follows from \cite[Lemma 1.1]{paques} that  $\nu(B) = A(V(B))$, so $\nu$ is surjective.
\end{proof}

\vu

Let $_{R-split}\mathfrak{Alg}$ denote the category whose objects are the $R$-split $K$-algebras  and whose morphisms are algebra homomorphisms. Also, let
$_{G-split}\mathfrak{FinSet}$ denote the category whose objects are finite $G$-split sets and whose morphisms are $G$-maps (i.e, maps that commute with the action of $G$).
Let $\theta:_{G-split}\mathfrak{FinSet}\to _{R-split}\mathfrak{Alg}$ and $\theta':_{R-split}\mathfrak{Alg}\to _{G-split}\mathfrak{FinSet}$
the maps given by $X\mapsto A(X)$ and $B\mapsto V(B)$, respectively.

\vu

\begin{teo} \label{gg}\, (\textbf{The Galois-Grothendieck equivalence})\,  Assume that $R$ is a $\beta$-Galois extension of $K$ and
$E_g$ is faithfully projective, for every $g\in G$. Then, $\theta$ is a contravariant functor that induces an equivalence between
the categories $_{G-split}\mathfrak{FinSet}$ and $_{R-split}\mathfrak{Alg}$, with inverse $\theta'$.
\end{teo}

\begin{proof} By Lemma \ref{xconj2}, given a finite $G$-split set $X$, the map $\varphi_g: E_g \otimes_K A(X) \longrightarrow \prod_{x \in X_g}E_{x}$ defined by
$\varphi_g(r \otimes_K f) = (rf(x))_{x \in X_g}$ is an isomorphism of $R$-algebras, for every $g\in G$. Thus, it is immediate, from the definitions,
that $V_g(A(X)) = V_g(X)$, for every $g \in G$. Indeed, it is enough to see that
$$\varphi_{(g,i)}(f) = \pi_{(g,i)}\varphi_g\eta_g(f) = \pi_{(g,i)}(\varphi_g(1_g \otimes_K f)) =\pi_{(g,i)}((f(x))_{x \in X_g}) = f(x) = \rho_x(f),$$
for all $f\in A(X)$ and $1\leq i\leq n_g$. Hence $V(X)=V(A(X))$.

\vd

Finally, recall that $X \simeq V(X)$ as $G$-sets, and $B \simeq A(V(B))$ as
$R^{\beta}$-algebras, by Lemmas \ref{xconj2}, \ref{lema1} and \ref{bijgro}. Hence, $X \simeq V(A(X)) = \theta'(\theta(X))$ and
$B \simeq A(V(B)) = \theta(\theta'(B))$.
\end{proof}

\vu


\section{The Galois-type correspondence}

Let $R$, $G$ and $\beta=\{\beta_g:E_{g^{-1}}\to E_g\ |\ g\in G\}$ be as in the previous section, and
$H\subseteq G$ an wide subgroupoid of $G$. Then, $\beta_H=\{\beta_h:E_{h^{-1}}\to E_h\ |\ h\in H\}$ is
an action of $H$ on $R$. Furthermore, recall from Example \ref{ex33} that $\frac{G}{H} = \{gH$ $| g \in G\}$ is a finite $G$-set via the action
$\gamma=\{\gamma_{g}:X_{g^{-1}}\to X_{g}\}_{g\in G}$, where $X_{g}= \{lH \in \frac{G}{H}|$ $r(l) = r(g)\}=X_{r(g)}$ and
$\gamma_{g}(lH) = glH$, for all $g\in G$. Recall also that $\frac{G}{H}=\dot\bigcup_{e\in G_0}X_{e}$.

\vu

\begin{lema}\label{isomgc} $A(\frac{G}{H}) \simeq R^{\beta_H}$ as $K$-algebras, for every wide subgroupoid $H$ of $G$.\end{lema}

\begin{proof}

We start by noticing that $\sum_{e \in G_0}f(eH)\in R^{\beta_H} $, for every $f\in A(\frac{G}{H})$.
Indeed, recall that $f(eH)\in E_e$, for all $e\in G_0$,  $\beta_{h^{-1}}(f(lH))=f(\gamma_{h^{-1}}(lH))=f(h^{-1}lH)$,
for all $lH\in X_h$, and $hH=r(h)H$, for all $h\in H$.  So,

$$
\begin{array}{ccl}
\beta_h(\sum_{e \in G_0}f(eH)1_{h^{-1}}) & \! = & \! \sum_{e \in G_0}\beta_h(f(eH)1_{h^{-1}}) = \beta_h(f(d(h)H))  \\
& \! = & \!\beta_h(f(h^{-1}hH))=\beta_h(\beta_{h^{-1}}(f(hH))) = \beta_{r(h)}(f(hH))  \\
& \! = & \! f(hH)$$ $$= f(r(h)H)= f(r(h)H)1_{r(h)}= f(r(h)H)1_h\\
& \! =&\!\sum_{e \in G_0}f(eH)1_h.
\end{array}
$$

Therefore, the map $$
\begin{array}{cccc}
\theta \ : & \! A(\frac{G}{H}) & \! \longrightarrow
& \! R^{\beta_H} \\
& \! f & \! \longmapsto
& \! \sum_{e \in G_0}f(eH).
\end{array}
$$
is well defined.

\vd

Conversely, given $g_1,g_2\in G$ and $r\in R^{\beta_H}$, if $g_1H=g_2H$ then $\beta_{g_1}(r1_{g_1^{-1}})=\beta_{g_2}(r1_{g_2^{-1}}).$
Indeed, from $g_1H=g_2H$ it follows that for any $h_1\in H$ there exists $h_2\in H$ such that $g_1h_1=g_2h_2$.
So, $g_1=g_1d(g_1)=g_1r(h_1)=g_1h_1h_1^{-1}= g_2h_2h_1^{-1}$. Furthermore, $E_{(g_2h_2h_1^{-1})^{-1}}=E_{h_1}=E_{(h_2h_1^{-1})^{-1}}$
and $E_{g_2^{-1}}=E_{h_2h_1^{-1}}$. Thus,

\[
\begin{array}{ccl}
\beta_{g_1}(r1_{g_1^{-1}}) & = & \beta_{g_2h_2h_1^{-1}}(r1_{h_1h_2^{-1}g_2^{-1}}) =\beta_{g_2}(\beta_{h_2h_1^{-1}}(r1_{h_1h_2^{-1}g_2^{-1}}))\\
& = & \beta_{g_2}(\beta_{h_2h_1^{-1}}(r1_{h_1h_2^{-1}g_2^{-1}})\beta_{h_2h_1^{-1}}(1_{h_1h_2^{-1}g_2^{-1}}))\\
&=&\beta_{g_2}(\beta_{h_2h_1^{-1}}(r1_{h_1h_2^{-1}})\beta_{h_2h_1^{-1}}(1_{h_1h_2^{-1}}))=\beta_{g_2}(r1_{h_2h_1^{-1}})\\
&=&\beta_{g_2}(r1_{g_2^{-1}})\
\end{array}
\]
Hence, the map

$$
\begin{array}{cccl}
\theta' \ : & \! R^{\beta_H} & \! \longrightarrow
& \! Map(\frac{G}{H},R), \\
& \! r & \! \longmapsto
& \! \theta'_r
\end{array}
$$
where $\theta'_r(lH) = \beta_l(r1_{l^{-1}})$, is well defined. In fact, $\theta'_r(gH)\in A(\frac{G}{H})$ since

$$
\begin{array}{ccl}
\alpha_{g}(\theta'_r1'_{g^{-1}})(lH) & \! = & \! \beta_{g} (\theta'_r1_{g^{-1}}'(\gamma_{g^{-1}}(lH))) = \beta_{g}(\theta'_r(g^{-1}lH)1_{g^{-1}})  \\
& \! = & \!\beta_{g}(\beta_{g^{-1}l}(r1_{l^{-1}g})) = \beta_{g}(\beta_{g^{-1}}(\beta_l(r1_{l^{-1}})))  \\
& \! = & \! \beta_{r(g)}(\beta_l(r1_{l^{-1}})) = \beta_l(r1_{l^{-1}})\\
& \! = & \!  \beta_l(r1_{l^{-1}})1_g = \theta'_r1'_g(lH),
\end{array}
$$
for all $g \in G$ such that $r(g) = r(l)$. If $r(g)\neq r(l)$ then $\alpha_{g}(\theta'_r1'_{g^{-1}})(lH)=0=\theta'_r1'_g(lH)$.

\vu

Clearly, $\theta$ and $\theta'$ are homomorphisms of $K$-algebras. Furthermore,
$$
\begin{array}{ccl}
\theta \circ \theta'(r) & \! = & \! \theta(\theta'_r) = \sum_{e \in G_0}\theta^{-1}_r(eH)  \\
& \! = & \! \sum_{e \in G_0}\beta_e(r1_e) = \sum_{e \in G_0}r1_e = r,
\end{array}
$$
for every $r\in R$, and
$$
\begin{array}{ccl}
\theta' \circ \theta(f)(gH) & \! = & \! \theta'_{\sum_{e \in G_0}f(eH)}(gH) = \beta_g(\sum_{e \in G_0}f(eH)1_{g^{-1}})  \\
& \! = & \! \beta_g(f(d(g)H)) = \beta_g(\beta_{g^{-1}}(f(gH)))  \\
& \! = & \! \beta_{r(g)}(f(gH)) = f(gH),
\end{array}
$$
for every $f\in A(\frac{G}{H})$ and $g\in G$. The proof is complete.
\end{proof}

\vd

For any $K$-subalgebra $T$ of $R$ put $H_T = \{g \in G \mid \beta_g(t1_{g^{-1}}) = t1_g,\ \text{for all}\ t\in T\}$. It is easy to check that $H_T$
is an wild subgroupoid of $G$. We say that  $T$ is $\beta$-\emph{strong} if for every $g, h \in G$ such that $r(g) = r(h)$ and $g^{-1}h \notin H_T$,
and, for every nonzero idempotent $e \in E_g = E_h$, there exists an element $t \in T$ such that $\beta_g(t1_{g^{-1}})e \neq  \beta_h(t1_{h^{-1}})e$.

\vu

\begin{lema}\label{isombeta} For each $gH \in \frac{G}{H}$, let $\rho_{gH}:A(\frac{G}{H})\to E_{r(g)}$  the homomorphism of $K$-algebras
given by $\rho_{gH}(f) = f(gH)$, for every $f \in A(\frac{G}{H})$. If the elements of $V_{gH} = \{\rho_{lH} |\quad r(l) = r(g)\}$ are pairwise
strongly distinct, then $R^{\beta_H}$ is $\beta$-strong.
\end{lema}

\begin{proof}By the Lemma \ref{isomgc}, $A(\frac{G}{H}) \simeq R^{\beta_H}$ via the map $\theta$. Consider
$\phi_{gH}:= \rho_{gH} \circ \theta^{-1}: R^{\beta_H} \to E_{r(g)}$. Since the elements of $V_{gH}$ are pairwise strongly distinct, it is easy to see
that the elements of $\widetilde{V}_{gH} = \{\phi_{lH} |\quad r(l) = r(g)\}$ are also pairwise strongly distinct.

Let $T = R^{\beta_H}$ and take $g, h \in G$ such that $r(g) = r(h)$ and $g^{-1}h \notin H_T$.
Given a nonzero idempotent $e \in E_g = E_h$, there exists $r \in R^{\beta_H}$ such that $\phi_{gH}(r)e \neq \phi_{hH}(r)e$. Thus,
$$
\begin{array}{ccl}
\beta_g(r1_{g^{-1}})e & \! = & \! \theta^{-1}(r)(gH)e = \rho_{gH}(\theta^{-1}(r))e = \phi_{gH}(r)e  \\
& \! \neq & \! \phi_{hH}(r)e = \rho_{hH}(\theta^{-1}(r))e = \theta^{-1}(r)(gH)e\\
& \! = & \! \beta_h(r1_{h^{-1}})e.
\end{array}
$$ Therefore, $R^{\beta_H}$ is $\beta$-strong.
\end{proof}

\vu

\begin{lema}\label{restricaot} Assume that $R$ is a $\beta$-Galois extension of $K$ and suppose that $T$ is a subalgebra of $R$ which is
separable over $K$ and $\beta$-strong. Then there exist elements $x_i, y_i \in T$, $1 \leq i \leq m$, such that
$\sum_{i=1}^mx_i\beta_g(y_i1_{g^{-1}}) = \delta_{e,g}1_e$, for all $e \in G_0$. In particular,
$T$ is a faithfully projective $K$-module.
\end{lema}

\begin{proof} Let $\upsilon = \sum_{i=1}^nx_i \otimes y_i \in T \otimes T$ be the separability idempotent of $T$ over
$K$ and $\mu: T \otimes T$ the multiplication map. For $g \in G$, define

$$
\begin{array}{cccc}
\psi_g \ : & \! T \otimes T  & \! \to
& \! T \otimes E_g \\
& \! x \otimes y & \! \mapsto
& \! x \otimes \beta_g(y1_{g^{-1}}).
\end{array}
$$
and take $\upsilon_g = \mu(\psi_g(e)) = \sum_{i=1}^nx_i\beta_g(y_i1_{g^{-1}}) \in E_g$. Clearly, $\upsilon_g$ is an idempotent of $E_g$,
for $\mu$ and $\psi$ are $K$-algebra homomorphisms. In particular, $\upsilon_e=1_e$, for all $e\in G_0$.

\vu

Moreover, $\mu$ and $\psi_g$ are $T \otimes K$-linear. Thus, for every $t \in T$,
$$
\begin{array}{ccl}
t\upsilon_g & \! = & \! t\mu(\psi_g(e)) = (t \otimes 1_R).\mu(\psi_g(e)) = \mu(\psi_g((t \otimes 1_R)e))  \\
& \! = & \! \mu(\psi_g((1_R \otimes t)e)) = \mu(\psi_g((1_R \otimes t))\mu(\psi_g(e))\\
& \! = & \! \beta_g(t1_{g^{-1}})\upsilon_g.
\end{array}
$$
Since $T$ is $\beta$-strong, if $g \notin G_0$, then $\upsilon_g = 0$, that is, $\sum_{i=1}^nx_i\beta_g(y_i1_{g^{-1}}) = 0$.

\vu

For the second part, it is enough to take the maps $f_i\in Hom_K(T,K)$ given by $f_i(t)=tr_{\beta}(y_it)$, $1\leq i\leq m$, and to
see that
$$\sum_{i=1}^nf_i(t)x_i = \sum_{i=1}^n\sum_{g \in G}\beta_g(y_it1_{g^{-1}})x_i = \sum_{e \in G_0}1_et = 1_Rt=t,$$
for every  $t\in T$.
\end{proof}

\vu

\begin{lema}\label{rdecomp} Assume that $R$ is a $\beta$-Galois extension of $K$ and let $T$ be a subalgebra of $R$.
Then the following conditions are equivalents:

\begin{itemize}

\vu

\item [(i)] $T$ is separable over $K$ and $\beta$-strong;

\vu

\item [(ii)] $T = R^{\beta_{H_T}}$.
\end{itemize}
In particular, in this case, $T$ is $R$-split.
\end{lema}

\begin{proof}(i) $\Rightarrow$ (ii) By Lemma \ref{restricaot}, $T$ is projective and finitely generated as $K$-module.
Since $T \subseteq R^{\beta_{H_T}}$, we have $T_{\mathfrak{p}} \subseteq (R^{\beta_{H_T}})_{\mathfrak{p}}$, and thus
$rank_{K_{\mathfrak{p}}} T_{\mathfrak{p}} \leq rank_{K_{\mathfrak{p}}} (R^{\beta_{H_T}})_{\mathfrak{p}}$,
for every prime ideal $\mathfrak{p}$ of $K$. We shall prove that indeed
$rank_{K_{\mathfrak{p}}} T_{\mathfrak{p}} = rank_{(K_{\mathfrak{p}}} (R^{\beta_{H_T}})_{\mathfrak{p}}$ for every prime
ideal $\mathfrak{p}$ of $K$, and, consequently, $T = R^{\beta_{H_T}}$, by \cite[Lemma 1.1]{paques}.

\vu

Let $\{g_i\in G\ | 1\leq i \leq n\}$ be a left tranversal of $H_T$ in $G$. Define
$$
\begin{array}{cccc}
f_i \ : & \! T  & \! \longrightarrow
& \! E_{g_i} \\
& \! t & \! \longmapsto
& \! \beta_{g_i}(t1_{g_i^{-1}}).
\end{array}
$$
Clearly, the $f_i$'s are $K$-algebra homomorphisms and the elements of $V_{g_i} = \{f_j\ |$ $1_{g_j} = 1_{g_i}\}$
are pairwise strongly distinct, for $T$ is $\beta$-strong. Therefore, by  Corollary \ref{lemacoro},
$\#V_{g_i} \leq rank_{(R^{\beta})_{\mathfrak{p}}} T_{\mathfrak{p}}$, for every prime ideal $\mathfrak{p}$ of $K$.

\vu

By  Lemma \ref{xconj2}, we have that  $E_{g_i} \otimes R^{\beta_{H_T}} \simeq
\prod_{x \in (\frac{G}{H_T})_{g_i}} E_{x}$, thus $(E_{g_i})_{\mathfrak{p}} \otimes_{K_{\mathfrak{p}}} (R^{\beta_{H_T}})_{\mathfrak{p}} \simeq
\prod_{x \in (\frac{G}{H_T})_{g_i}} (E_{x})_{\mathfrak{p}}$. Recall from Example \ref{ex33} that $(\frac{G}{H_T})_{g_i} = \{lH_T|$ $r(l) = r(g_i)\}$.
Then, $\#V_{g_i} = \#(\frac{G}{H_T})_{g_i}$. Therefore,

$$
\begin{array}{ccl}
rank_{K_{\mathfrak{p}}}(R^{\beta_{H_T}})_{\mathfrak{p}} & \! = & \! rank_{(E_{g_i})_{\mathfrak{p}}}((E_{g_i})_{\mathfrak{p}}
\otimes_{K_{\mathfrak{p}}} (R^{\beta_{H_T}})_{\mathfrak{p}})\\
& \! = & \! rank_{(E_{g_i})_{\mathfrak{p}}} \prod_{x \in (\frac{G}{H_T})_{g_i}}(E_{x})_{\mathfrak{p}}\\
& \! = & \! \#(\frac{G}{H_T})_{g_i} = \#V_{g_i} \leq rank_{(R^{\beta})_{\mathfrak{p}}} T_{\mathfrak{p}},
\end{array}
$$
and so $rank_{K_{\mathfrak{p}}} T_{\mathfrak{p}} = rank_{K_{\mathfrak{p}}} (R^{\beta_{H_T}})_{\mathfrak{p}}$.

\vd

(ii) $\Rightarrow$ (i) By Lemmas \ref{xconj2} and \ref{lema23}, $T = R^{\beta_{H_T}} \simeq A(\frac{G}{H_T})$ is separable over $K$.
Furthermore, by the Lemma \ref{lema1} the elements of $V_{gH_T}$ are pairwise strongly distinct. Hence,  $T$ is $\beta$-strong,
by  Lemma \ref{isombeta}.

The last assertion follows from Lemmas \ref{xconj2} and \ref{lema1}.
\end{proof}

\begin{teo} (\textbf{The Galois correspondence}) Assume that $R$ is a $\beta$-Galois extension of $K$ and $E_g$ is faithfully
projective, for every $g\in G$. Then the correspondence $H \mapsto R^{\beta_H}$ is one-to-one between the set of all the wide subgroupoids of $G$ and the
set of all the subalgebras of $R$ which are separable over $K$ and $\beta$-strong.
\end{teo}

\begin{proof}Let ${\bf wsg}(G)$ be the set of the {\bf w}ide {\bf s}ub{\bf g}roupoids $H$ of $G$, ${\bf quot}(G)$ the
set of the {\bf quot}ients sets $\frac{G}{H}$ of $G$ and ${\bf sss}(R)$ the set of the {\bf s}eparable and
$\beta$-{\bf s}trong $K$-{\bf s}ubalgebras of $R$. The bijection between ${\bf wsg}(G)$ and ${\bf quot}(G)$ is obvious. The
bijection between ${\bf quot}(G)$ and ${\bf sss}(R)$ follows from Lemma \ref{rdecomp} and Theorem \ref{gg}.
\end{proof}

\vu


\section{A final remark}

Again, $R$, $G$ and $\beta$ are as in the previous sections. In almost all results in the two last sections the assumption that
$E_g$ is a faithful $K$-module was required. So, it is natural to ask under what conditions such an assumption occurs. To answer
this question it is necessary to have a description of the elements in $R^\beta=K$. An easy calculus shows that an element $x=\sum_{e\in G_0}x_e\in
R=\bigoplus_{e\in G_0}E_e$ is in $R^\beta$ if and only if $x_{r(g)}=\beta_g(x_{d(g)})$, for all $g\in G$. It is an immediate consequence
of this fact that, given $x\in K$ and $g\in G$, $x1_g=0$ if and only if $x_{r(g)}=0$ if and only if $x_{d(g)}=0$. Therefore, given $x\in K$ and
$g\in G$, $x1_g=0$ implies $x=0$ if and only if, for all $h\in G$,  either $d(h^{\pm 1})=d(g)$ or
$d(h^{\pm 1})=r(g)$. From these considerations we have the following lemma.

\vu

\begin{lema}\label{lema} For each $g\in G$, $E_g$ is faithful over $K$ if and only if either $d(h^{\pm 1})= d(g)$ or $d(h^{\pm 1})=r(g)$,
for all $h\in G$.\end{lema}

\vu

The following two examples illustrate the above lemma. Notice that both of them are also examples of $\beta$-Galois extensions.

\vu

\begin{exes} \em
(1) Consider $R=Sv_1\oplus Sv_2\oplus Sv_3\oplus Sv_4$, where $S$ is a ring and $v_1,v_2,v_3$ and $v_4$ are pairwise orthogonal central
idempotents of $R$, with sum $1_R$. Let  $G=\{g,g^{-1},d(g),r(g)\}$ be a groupoid and $\beta$ the action of $G$ on $R$ given by: $E_g=E_{r(g)}=
Sv_3\oplus Sv_4$, $E_{g^{-1}}=E_{d(g)}=Sv_1\oplus Sv_2$, $\beta_{r(g)}=I_{E_{r(g)}}$, $\beta_{d(g)}=I_{E_{d(g)}}$, $\beta_g(av_1+bv_2)=av_3+bv_4$,
$\beta_{g^{-1}}(av_3+bv_4)=av_1+bv_2$, for all $a,b\in S$. It is easy to see that $R$ is a $\beta$-Galois extension of
$K=R^\beta=S(v_1+v_3)\oplus S(v_2+v_4)$, with $\beta$-Galois coordinates $x_i=v_i=y_i$, $1\leq i\leq 4$. Furthermore, it is immediate that
$xE_g=0=xE_{g^{-1}}$ if and only if $x=0$, for all $x\in K$.

\vd

(2) Let $R=Sv_1\oplus Sv_2\oplus Sv_3\oplus Sv_4\oplus Sv_5\oplus Sv_6$, where $S$ is a ring and $v_i$, $1\leq i\leq 6$, are pairwise orthogonal central
idempotents of $R$, with sum $1_R$. Take the groupoid $G=\{g,g^{-1}, d(g), r(g), h=h^{-1}, d(h)=r(h)\}$ and $\beta=\{\beta_l:E_{l^{-1}}\to E_l\}_{l\in G}$,
where $E_l$ and $\beta_l$, for $l=d(g),r(g),g,g^{-1}$, are as in the example (1), $E_h=E_{r(h)}=Sv_5+Sv_6$, $\beta_{r(h)}=I_{E_{r(h)}}$, and
$\beta_h(av_5+bv_6)=av_6+bv_5$. Again, $R$ is a $\beta$-Galois extension of $K=R^\beta=S(v_1+v_3)\oplus S(v_2+v_4)\oplus S(v_5+v_6)$, with $\beta$-Galois
coordinates $x_i=v_i=y_i$, $1\leq i\leq 6$. Nevertheless, in this case we have, for instance, $xE_h=0$ for $x=v_1+v_3\in K$.
\end{exes}

\vd

\bibliographystyle{amsalpha}
{
}

\end{document}